\documentclass[11pt, reqno]{amsart}
\usepackage{amsfonts}
\usepackage{amsmath}
\usepackage{amssymb, color}
\usepackage{amscd}
\usepackage[mathscr]{eucal}
\usepackage{url}
\usepackage{enumerate}

\allowdisplaybreaks[1]

\renewcommand*{\bar}{\overline}

\theoremstyle{plain}

\newtheorem{theorem}{Theorem}[section]
\newtheorem{lemma}[theorem]{Lemma}
\newtheorem{prop}[theorem]{Proposition}
\newtheorem{cor}[theorem]{Corollary}
\theoremstyle{definition}

\numberwithin{equation}{section}

\newcommand{\outdeg}{\textup{outdeg}}

\newcommand{\indeg}{\textup{indeg}}

\makeatletter
\def\imod#1{\allowbreak\mkern5mu({\operator@font mod}\,\,#1)}
\makeatother

\makeatletter
\@namedef{subjclassname@2020}{%
  \textup{2020} Mathematics Subject Classification}
\makeatother

\begin{document}

\title[Paley Digraphs]{Transitive subtournaments of $k$-th Power\\Paley Digraphs and improved lower\\ bounds for Ramsey numbers}

\author{Dermot M\lowercase{c}Carthy, Mason Springfield}

\address{Dermot M\lowercase{c}Carthy, Department of Mathematics \& Statistics, Texas Tech University, Lubbock, TX 79410-1042, USA}
\email{dermot.mccarthy@ttu.edu}
\urladdr{https://www.math.ttu.edu/~mccarthy/}

\address{Mason Springfield, Department of Mathematics \& Statistics, Texas Tech University, Lubbock, TX 79410-1042, USA}
\email{mason.springfield@ttu.edu}


\begin{abstract}
Let $k \geq 2$ be an even integer. Let $q$ be a prime power such that $q \equiv k+1 \imod {2k}$. We define the \emph{k-th power Paley digraph} of order $q$, $G_k(q)$, as the graph with vertex set $\mathbb{F}_q$ where $a \to b$ is an edge if and only if $b-a$ is a $k$-th power residue. This generalizes the (k=2) Paley Tournament. We provide a formula, in terms of finite field hypergeometric functions, for the number of transitive subtournaments of order four contained in $G_k(q)$, $\mathcal{K}_4(G_k(q))$, which holds for all $k$. We also provide a formula, in terms of Jacobi sums, for the number of transitive subtournaments of order three contained in $G_k(q)$, $\mathcal{K}_3(G_k(q))$. 
In both cases, we give explicit determinations of these formulae for small $k$. 
We show that zero values of $\mathcal{K}_4(G_k(q))$ (resp.~$\mathcal{K}_3(G_k(q))$) yield lower bounds for the multicolor directed Ramsey numbers $R_{\frac{k}{2}}(4)=R(4,4,\cdots,4)$ (resp.~$R_{\frac{k}{2}}(3)$). We state explicitly these lower bounds for $k\leq 10$ and compare to known bounds, showing improvement for $R_2(4)$ and $R_3(3)$. Combining with known multiplicative relations we give improved lower bounds for $R_{t}(4)$, for all $t\geq 2$, and for $R_{t}(3)$, for all $t \geq 3$.
\end{abstract}

\maketitle


\section{Introduction}\label{sec_Intro}
The Paley graphs are a well-known family of self-complementary strongly-regular undirected graphs. Let $\mathbb{F}_q$ denote the finite field with $q$ elements. For $q\equiv 1 \pmod 4$, the Paley graph of order $q$ is the graph with vertex set $\mathbb{F}_q$ where $ab$ is an edge if and only if $a-b$ is a square. 
One of the earliest appearances of Paley graphs (although not called that at the time) in the literature was in Greenwood and Gleason's proof that the two-color diagonal Ramsey number $R(4,4)=18$, in 1955 \cite{GG}. They showed that the Paley graph of order 17 does not contain a clique of order four, thus showing $17<R(4,4)$, and then combined this with elementary upper bounds. Paley graphs can be generalized from connections based on squares to $k$-th powers, for any integer $k \geq 2$, provided  $q \equiv 1 \imod {k}$ if $q$ is even, or, $q \equiv 1 \imod {2k}$ if $q$ is odd \cite{C, LP}. Finding the number of cliques of a given order and improving bounds for the order of the maximum clique (i.e., the clique number) in generalized Paley graphs is an open problem and an active area of inquiry \cite{CL, DM, GSY, HP, Y, Y2}. Using various types of character sums to count cliques in generalized Paley graphs and other similar types of graphs is a common theme \cite{BB, BB2, BB3, DM, Y}.

When $q\equiv 3 \pmod 4$, we can use a similar construction to Paley graphs to define the Paley tournament. Specifically, the Paley tournament of order $q$ is the digraph with vertex set $\mathbb{F}_q$ where $a \to b$ is an edge if and only if $b-a$ is a square. Erd\"{o}s and Moser \cite{EM} used the Paley tournament of order seven to prove that the directed Ramsey number $R(4)=8$. They showed that the Paley tournament of order seven does not contain a transitive subtournament of order four, thus showing $7<R(4)$, and then combined this with elementary upper bounds. 

In a similar manner to how Paley graphs can be generalized, we can generalize the Paley tournament construction to higher powers.
Let $k \geq 2$ be an even integer. Let $q$ be a prime power such that $q \equiv k+1 \imod {2k}$. Let $S_k$ be the subgroup of the multiplicative group $\mathbb{F}_q^{\ast}$ of order $\frac{q-1}{k}$ containing the $k$-th power residues, i.e., if $\omega$ is a primitive element of $\mathbb{F}_q$, then $S_k = \langle \omega^k \rangle$. Then we define the \emph{k-th power Paley digraph} of order $q$, $G_k(q)$, as the graph with vertex set $\mathbb{F}_q$ where $a \to b$ is an edge if and only if $b-a \in S_k$. We note, due to the conditions imposed on $q$, that $-1 \notin S_k$ so $G_k(q)$ is a well-defined directed graph. When $k=2$ we recover the Paley tournament.
There is little reference to these graphs in the literature for $k>2$. Ananchuen \cite{An} proves that, for positive integers $n$ and $t$, and for $q>f(n,t)$ sufficiently large, $G_4(q)$ has the property that every subset of $n$ vertices is dominated by at least $t$ other vertices. Podest\'{a} and Videla \cite{PV} study the spectral properties of $G_k(q)$ and give an explicit evaluation of the spectrum when $k=4$. 

Let $\mathcal{K}_m(G)$ denote the number of transitive subtournaments of order $m$ contained in a digraph $G$. The main purpose of this paper is to provide a general formula for both $\mathcal{K}_3(G_k(q))$ and $\mathcal{K}_4(G_k(q))$, which hold for all $k$. In both cases, we give explicit determinations of these formulae for small $k$. We also examine the consequences for lower bounds for multicolor directed Ramsey numbers. Specifically, we will show that if $\mathcal{K}_m(G_k(q))=0$ for some $q$, then $q<R_{\frac{k}{2}}(m)$ and use this to provide lower bounds for the multicolor directed Ramsey numbers $R_{\frac{k}{2}}(3)$ and $R_{\frac{k}{2}}(4)$. We compare to known bounds, showing improvement  in the case of $R_2(4)$ and $R_3(3)$. Combining these with known multiplicative relations, we give improved lower bounds for $R_{t}(4)$, for all $t\geq 2$, and for $R_{t}(3)$, for all $t \geq 3$.

We use similar techniques to \cite{DM}, which contains analogous results for generalized undirected Paley graphs. Besides the obvious differences of dealing with digraphs and transitive subtournaments in this paper, as opposed to undirected graphs and cliques in \cite{DM}, the character sum determinations are more complicated in this paper due to the fact that -1 is not a $k$-th power. 
This is especially relevant for the evaluations in Section \ref{sec_ProofThm2}.

\section{Statement of Main Results}\label{sec_Results}
\subsection{Subtournaments of Order Four}\label{sec_Results_Four}
Our most general results will be stated in terms of Greene's finite field hypergeometric function \cite{G, G2}. Let $\widehat{\mathbb{F}^{*}_{q}}$ denote the group of multiplicative characters of $\mathbb{F}^{*}_{q}$. We extend the domain of $\chi \in \widehat{\mathbb{F}^{*}_{q}}$ to $\mathbb{F}_{q}$, by defining $\chi(0):=0$ (including the trivial character $\varepsilon$) and denote $\bar{\chi}$ as the inverse of $\chi$. 
For $A, B \in \widehat{\mathbb{F}^{*}_{q}}$, we define the Jacobi sum $J(A, B):=\sum_{a \in \mathbb{F}_{q}} A(a) B(1-a)$ and define the symbol
$\binom{A}{B} := \frac{B(-1)}{q} J(A, \bar{B})$.
For characters $A_0,A_1,\dotsc, A_n$ and $B_1, \dotsc, B_n$ of $\mathbb{F}_{q}^*$ and 
$\lambda \in \mathbb{F}_{q}$, define the finite field hypergeometric function
\begin{equation*}
{_{n+1}F_n} {\left( \begin{array}{cccc} A_0, & A_1, & \dotsc, & A_n \\
\phantom{A_0} & B_1, & \dotsc, & B_n \end{array}
\Big| \; \lambda \right)}_{q}
:= \frac{q}{q-1} \sum_{\chi} \binom{A_0 \chi}{\chi} \binom{A_1 \chi}{B_1 \chi}
\dotsm \binom{A_n \chi}{B_n \chi} \chi(\lambda),
\end{equation*}
where the sum is over all multiplicative characters $\chi$ of $\mathbb{F}_{q}^*$.
For $k \geq 2$ an integer, let $\chi_k \in \widehat{\mathbb{F}^{*}_{q}}$ be a character of order $k$, when $q\equiv 1 \imod {k}$.
For $\vec{t}=(t_1,t_2,t_3,t_4,t_5) \in \left({\mathbb{Z}_{k}}\right)^{5}$, we define
\begin{equation*}
{_{3}F_2} \left( \vec{t} \; \big| \;  \lambda \right)_{q,k}
:=
(-1)^{t_3+t_5} {_{3}F_2}\biggl( \begin{array}{ccc} \chi_k^{t_1}, & \chi_k^{t_2}, & \chi_k^{t_3} \vspace{.05in}\\
\phantom{\chi_k^{t_1}} & \chi_k^{t_4}, & \chi_k^{t_5} \end{array}
\Big| \; \lambda \biggr)_{q}.
\end{equation*}

\begin{theorem}\label{thm_Main1}
Let $k \geq 2$ be an even integer. Let $q$ be a prime power such that $q \equiv k+1 \imod {2k}$. Then
\begin{equation*}
\mathcal{K}_4(G_k(q)) = 
\frac{q(q-1)}{k^6} 
\sum_{\vec{t} \in \left({\mathbb{Z}_{k}}\right)^{5}} {_{3}F_2} \left( \vec{t} \; \big| \;  1 \right)_{q,k}.
\end{equation*}
\end{theorem}
\noindent 
We can use known reduction formulae for finite field hypergeometric functions to simplify many of the summands in Theorem \ref{thm_Main1}.
Evaluating these terms yields our second result. 

\begin{theorem}\label{thm_Main2}
Let $k \geq 2$ be an even integer. Let $q$ be a prime power such that $q \equiv k+1 \imod {2k}$.  Then
\begin{multline*}
\mathcal{K}_4(G_k(q)) =
\frac{q(q-1)}{k^6} 
\Biggl[ 
10 \, \mathbb{R}_k(q)^2 + 5 \left( q-k^2+1\right) \mathbb{R}_k(q)  -10 \, \mathbb{S}_k(q) -5 \, \mathbb{S}^{-}_k(q) 
\\
+q^2-10 \left(k-1 \right)^2 q
+ 5k^2(k-1)+1
+q^2 \sum_{\vec{t} \in X_k} {_{3}F_2} \left( \vec{t} \; \big| \;  1 \right)_{q,k}
\Biggr],
\end{multline*}
where $X_k := \{  (t_1,t_2,t_3,t_4,t_5) \in \left({\mathbb{Z}_{k}}\right)^{5} \mid t_1,t_2,t_3 \neq 0, t_4,t_5 \, ; \, t_1+t_2+t_3 \neq t_4+t_5 \}$,
\begin{equation*}
\mathbb{R}_k(q) := \sum_{\substack{s,t=1 \\s+t \not\equiv 0 \, (k)}}^{k-1} 
J \left( \chi_k^s, \chi_k^t \right),
\qquad \qquad
\mathbb{S}_k(q) := 
\sum_{\substack{s,t,v=1 \\ s+t, v+t, v-s \not\equiv 0 \, (k) }}^{k-1} 
J \left( \chi_k^s, \chi_k^t \right)
J \left(\bar{\chi_k}^s , \chi_k^v \right).
\end{equation*}
and
\begin{equation*}
\mathbb{S}^{-}_k(q) := 
\sum_{\substack{s,t,v=1 \\ s+t, v+t, v-s \not\equiv 0 \, (k) }}^{k-1} 
(-1)^{s+t}
J \left( \chi_k^s, \chi_k^t \right)
J \left(\bar{\chi_k}^s , \chi_k^v \right).
\end{equation*}

\end{theorem}
\noindent
Many of the summands that still remain in Theorem \ref{thm_Main2} are equal, up to sign.
A group action on $X_k$ is described in Section \ref{sec_Orbits}, which allows us to restrict the sum to orbit representatives.
For specific small $k$, Theorem \ref{thm_Main2} reduces to relatively few terms.

\begin{cor}[${k=2}$]\label{cor_k2}
Let $q \equiv 3 \pmod {4}$ be a prime power. Then
\begin{equation*}
\mathcal{K}_4(G_2(q)) =
\frac{q(q-1)(q-3)(q-7)}{2^6}. 
\end{equation*}
\end{cor}
\noindent
So, $\mathcal{K}_4(G_2(7)) =0$ and we reconfirm the lower bound $8 \leq R(4)$ of Erd\"{o}s and Moser.

\begin{cor}[${k=4}$]\label{cor_k4}
Let $q=p^r\equiv 5 \pmod {8}$ for a prime $p$. Let $\varphi, \chi_4  \in \widehat{\mathbb{F}^{*}_{q}}$ be characters of order $2$ and $4$ respectively.
Write $q=x^2 +y^2$ for integers $x$ and $y$, such that $x \equiv 1 \pmod{4}$, and $p \nmid x$.
Then
\begin{multline*}
\mathcal{K}_4(G_4(q)) =
\frac{q(q-1)}{2^{12}}. 
\Biggl[ 
q^2+2q(5x-21)
+24x^2-150x+241\\
+10 \, q^2 
{_{3}F_2} {\left( \begin{array}{cccc}  \chi_4, & \varphi, &  \varphi \\
\phantom{\chi_4} & \varepsilon, &  \varepsilon \end{array}
\Big| \; 1 \right)}_{q}
\Biggr].
\end{multline*}
\end{cor}
\noindent 
An algorithm for determining $x$ can be found in \cite{McC11}.
It is interesting to note that, when $q=p$ is prime, the values of the hypergeometric function in Corollary \ref{cor_k4} correspond to the $p$-th Fourier coefficients of a certain non-CM modular form of weight three and level 32 \cite{DM2, MP}.

We will discuss the relationship between the $k$-th power Paley digraphs and multicolor directed Ramsey numbers in Section \ref{sec_Ramsey}. Specifically, we will show that if $\mathcal{K}_m(G_k(q))=0$ for some $q$, then $q<R_{\frac{k}{2}}(m)$. For $k=4$ and $q=5^3$ we get that $x=-11$ and 
$q^2 
{_{3}F_2} {\left( \begin{array}{cccc}  \chi_4, & \varphi, &  \varphi \\
\phantom{\chi_4} & \varepsilon, &  \varepsilon \end{array}
\Big| \; 1 \right)}_{q}
=-142$
and so, by Corollary \ref{cor_k4}, $\mathcal{K}_4(G_4(125)) = 0$. 
Thus, $126 \leq R(4,4)$.
Based on a review of the literature, this seems to be an improvement on the best known lower bound.
Combining with results from \cite{MT} we get the following improved lower bounds for $R_t(4)$ in general. 
See Section \ref{sec_Ramsey} for further details.
\begin{cor}\label{cor_Ramsey4t}
For $t\geq 2$,
$$125 \cdot  7^{t-2} +1 \leq R_t(4).$$
\end{cor}
\noindent


\subsection{Subtournaments of Order Three}\label{sec_Results_Three}
We have similar results for $\mathcal{K}_3(G_k(q))$ in terms of Jacobi sums.
\begin{theorem}\label{thm_Main3}
Let $k \geq 2$ be an even integer. Let $q$ be a prime power such that $q \equiv k+1 \imod {2k}$. Then
\begin{equation*}
\mathcal{K}_3(G_k(q)) = 
\frac{q(q-1)}{k^3} \, ( \mathbb{R}_k(q) + q -2k+1),
\end{equation*}
where $\mathbb{R}_k(q)$ is as defined in Theorem \ref{thm_Main2}.
\end{theorem}
\noindent

\begin{cor}[${k=2}$]\label{cor_3k2}
Let $q\equiv 3 \pmod {4}$ be a prime power. Then
\begin{equation*}
\mathcal{K}_3(G_2(q)) =
\frac{q(q-1)(q-3)}{2^3}. 
\end{equation*}
\end{cor}

\begin{cor}[${k=4}$]\label{cor_3k4}
Let $q =p^r \equiv 5 \pmod {8}$ for a prime $p$.
Write $q=x^2 +y^2$ for integers $x$ and $y$, such that $x \equiv 1 \pmod{4}$, and $p \nmid x$.
Then
\begin{equation*}
\mathcal{K}_3(G_4(q)) =
\frac{q(q-1)(q+2x-7)}{2^{6}}. 
\end{equation*}
\end{cor}
\noindent 
It is easy to see from Corollaries \ref{cor_3k2} and \ref{cor_3k4} that $\mathcal{K}_3(G_2(3)) = \mathcal{K}_3(G_4(13)) = 0$ which leads to the following corresponding lower bounds for multicolor Ramsey numbers:
$4 \leq R(3)$ and
$14 \leq R(3,3)$.
\noindent
It is known that $R(3)=4$ and $R(3,3)=14$ \cite{BD, MT}. 
However, when $k=6$ we get $44 \leq R_3(3) = R(3,3,3)$, which seems to be an improvement on the best known lower bound from the literature. Combining with results from \cite{MT} we get the following improved lower bounds for $R_t(3)$ in general. See Section \ref{sec_Ramsey} for further details.
\begin{cor}\label{cor_Ramsey3t}
For $t \geq 3$,
$$43 \cdot 3^{t-3} +1  \leq R_t(3).$$
\end{cor}


\section{Preliminaries}\label{sec_Prelim}

\subsection{Jacobi Sums}\label{sec_Prelim_GJSums}
We start by recalling some well-known properties of Jacobi sums. See \cite{BEW} for more details, noting that we have adjusted results therein to account for $\varepsilon(0)=0$. 

\begin{prop}\label{prop_JacBasic}
For non-trivial $\chi \in \widehat{\mathbb{F}^{*}_{q}}$ we have
\begin{itemize}
\item[(a)] $J(\varepsilon,\varepsilon)=q-2$;
\item[(b)] $J(\varepsilon, \chi)=-1$; and
\item[(c)] $J(\chi,\bar{\chi}) = -\chi(-1)$.
\end{itemize}
\end{prop}

\begin{prop}\label{prop_JacXfer}
For $\chi, \psi \in \widehat{\mathbb{F}^{*}_{q}}$,
$J(\chi, \psi) = \chi(-1) J(\chi, \bar{\chi} \bar{\psi})$.
\end{prop}

\begin{prop}\label{prop_JacProdq}
For non-trivial $\chi, \psi \in \widehat{\mathbb{F}^{*}_{q}}$ with $\chi \psi$ non-trivial,
$J(\chi, \psi) J(\bar{\chi}, \bar{\psi}) =q$.
\end{prop}

\noindent
Recall, if we let $k \geq 2$ be an integer, $q\equiv 1 \pmod {k}$ be a prime power and $\chi_k \in \widehat{\mathbb{F}^{*}_{q}}$ be a character of order $k$,
then for $b \in \mathbb{F}_q^*$, we have the orthogonal relation \cite[p11]{BEW}
\begin{equation}\label{for_OrthRel}
\frac{1}{k} \sum_{t=0}^{k-1} \chi_k^t(b)
=
\begin{cases}
1 & \textup{ if $b$ is a $k$-th power},\\
0 & \textup{ if $b$ is not a $k$-th power}.
\end{cases}
\end{equation}

We now develop some preliminary results for later use. In addition to $\mathbb{R}_k(q)$, $\mathbb{S}_k(q)$ and $\mathbb{S}^{-}_k(q)$ defined in Theorem \ref{thm_Main2}, we define 
\begin{equation*}
\mathbb{R}^{-}_k(q) := 
\sum_{\substack{s,t=1 \\s+t \not\equiv 0 \, (k)}}^{k-1} 
(-1)^{s+t}
J \left( \chi_k^s, \chi_k^t \right).
\end{equation*}
In Section \ref{sec_ProofThm2}, we will simplify many expressions involving Jacobi sums.  Many of these expressions appear multiple times so we name them and list their simplifications here for ease of reference.
\begin{prop}\label{prop_JtoRS}
Let $k \geq 2$ be an even integer, $q$ be a prime power such that $q \equiv k+1 \imod {2k}$ and $\chi_k \in \widehat{\mathbb{F}^{*}_{q}}$ be a character of order $k$.
\begin{itemize}
\item[(a)] 
$\mathbb{J}_{0}(q,k) :=
\displaystyle \sum_{s,t=0}^{k-1} 
J \left( \chi_k^s, \chi_k^t \right)
=
\mathbb{R}_k(q) + q -2k + 1$;
\item[(b)]
$\mathbb{J}^{-}_{0}(q,k) :=
\displaystyle \sum_{s,t=0}^{k-1} 
(-1)^{s+t}
J \left( \chi_k^s, \chi_k^t \right)
= 
\mathbb{R}^{-}_k(q) + q + 1$;
\item[(c)]
$
\mathbb{JJ}_{0}(q,k) :=
\displaystyle \sum_{s,t,v=0}^{k-1} 
J \left( \chi_k^s, \chi_k^t \right)
J \left(\bar{\chi_k}^s , \chi_k^v \right)
=
\mathbb{S}_k(q) -4 \mathbb{R}_k(q) + q^2 +q(k^2-5k) + k^2+4k-3; 
$
\item[(d)]
$\mathbb{JJ}^{-}_0(q,k):=
\displaystyle \sum_{s,t,v=0}^{k-1} 
(-1)^{s+t}
J \left( \chi_k^s, \chi_k^t \right)
J \left(\bar{\chi_k}^s , \chi_k^v \right)
=
\mathbb{S}^{-}_k(q)  - \mathbb{R}^{-}_k(q)  - 3 \mathbb{R}_k(q) + q^2 -2kq +3(k-1)$;
\item[(e)] 
$\displaystyle \sum_{s,t=0}^{k-1} 
(-1)^{s}
J \left( \chi_k^s, \chi_k^t \right)
= 
\mathbb{J}_{0}(q,k) $;
\item[(f)]
$\displaystyle \sum_{s,t,v=0}^{k-1} 
(-)^{t+v}
J \left( \chi_k^s, \chi_k^t \right)
J \left(\bar{\chi_k}^s , \chi_k^v \right)
=
\mathbb{JJ}_{0}(q,k) +k^2(q-1)$;
\item[(g)]
$\displaystyle \sum_{\substack{s,t,v=1 \\ s+t, v+t, v-s \not\equiv 0 \, (k) }}^{k-1} 
(-1)^{t+v}
J \left( \chi_k^s, \chi_k^t \right)
J \left(\bar{\chi_k}^s , \chi_k^v \right)
=
\mathbb{S}_k(q) + qk(k-2)$;
\end{itemize}
\end{prop}

\begin{proof}
All are a relatively straightforward consequence of Propositions \ref{prop_JacBasic}-\ref{prop_JacProdq}.
\end{proof}

\begin{lemma}[\cite{DM, KR}]\label{lem_JacOrder4}
Let $q=p^r\equiv 1 \pmod {4}$ for a prime $p$. Write $q=x^2 +y^2$ for integers $x$ and $y$, such that $x\equiv 1 \pmod{4}$, and $p \nmid x$ when $p \equiv 1 \imod{4}$. Then
\begin{enumerate}
\item $J(\chi_4,\chi_4) + J(\bar{\chi_4},\bar{\chi_4}) = -2x$; and
\item $J(\chi_4,\chi_4)^2 + J(\bar{\chi_4},\bar{\chi_4})^2 = 2x^2-2y^2 = 4x^2-2q = 2q-4y^2$.
\end{enumerate}
\end{lemma}


\subsection{Properties of Finite Field Hypergeometric Functions}\label{sec_Prelim_HypFns}
Our most general results from Section \ref{sec_Results} are given in terms of Greene's finite field hypergeometric functions.
These functions can be expressed as character sums in a simple way \cite[Def 3.5 (after change of variable), Cor 3.14]{G2}.
For characters $A, B, C, D, E$ of $\mathbb{F}_{q}^*$,
\begin{equation}\label{for_CharSum2F1}
q \, {_{2}F_1} {\left( \begin{array}{cc} A, & B \\
\phantom{A} & C \end{array}
\Big| \; \lambda \right)}_{q}
= \sum_{b \in \mathbb{F}_q} A\bar{C}(b) \bar{B}C(1-b)  \bar{A} (b- \lambda)
\end{equation}
and
\begin{equation}\label{for_CharSum3F2}
q^2 \, {_{3}F_2} {\left( \begin{array}{ccc} A, & B, &  C \\
\phantom{A} & D, & E \end{array}
\Big| \; \lambda \right)}_{q}
=  \sum_{a,b \in \mathbb{F}_q} A\bar{E}(a) \bar{C}E(1-a) B(b) \bar{B}D(b-1) \bar{A} (a- \lambda b)
\end{equation}
The following reduction formulae \cite[Thms 3.15 \& 4.35]{G2} will play an important part in proving Theorem \ref{thm_Main2}.
\begin{align}
\label{for_3F2Red_1}
{_{3}F_2} {\left( \begin{array}{ccc} \varepsilon, & B, &  C \\
\phantom{A} & D, & E \end{array}
\Big| \; 1 \right)}_{q}
=&
-\frac{1}{q} \,
{_{2}F_1} {\left( \begin{array}{cc} B\bar{D}, &  C\bar{D} \\
\phantom{B\bar{D}} & E\bar{D} \end{array}
\Big| \; 1 \right)}_{q}
+ \binom{B}{D} \binom{C}{E};
\\[6pt] \label{for_3F2Red_2}
{_{3}F_2} {\left( \begin{array}{ccc} A, & \varepsilon, &  C \\
\phantom{A} & D & E \end{array}
\Big| \; 1 \right)}_{q}
=&
A(-1) \binom{D}{A}
{_{2}F_1} {\left( \begin{array}{cc} A\bar{D}, &  C\bar{D} \\
\phantom{A\bar{D}} & E\bar{D} \end{array}
\Big| \; 1 \right)}_{q}
-\frac{D(-1)}{q} \binom{C}{E};
\\[6pt] \label{for_3F2Red_3}
{_{3}F_2} {\left( \begin{array}{ccc} A, & B, &  C \\
\phantom{A} & A, & E \end{array}
\Big| \; 1 \right)}_{q}
=&
\binom{B}{A}
{_{2}F_1} {\left( \begin{array}{cc} B, &  C\\
\phantom{B} & E \end{array}
\Big| \; 1 \right)}_{q}
-\frac{\bar{A}(-1)}{q} \binom{C\bar{A}}{E\bar{A}};
\\[6pt] \label{for_3F2Red_4}
{_{3}F_2} {\left( \begin{array}{ccc} A, & B, &  C \\
\phantom{A} & B, & E \end{array}
\Big| \; 1 \right)}_{q}
=&
-\frac{1}{q} \,
{_{2}F_1} {\left( \begin{array}{cc} A, &  C \\
\phantom{A} & E \end{array}
\Big| \; 1 \right)}_{q}
+ \binom{A\bar{B}}{\bar{B}} \binom{C\bar{B}}{E\bar{B}};
\\[6pt] \label{for_3F2Red_5}
{_{3}F_2} {\left( \begin{array}{ccc} A, & B, &  C \\
\phantom{A} & D, & B \end{array}
\Big| \; 1 \right)}_{q}
=&
\binom{C\bar{D}}{B\bar{D}}
{_{2}F_1} {\left( \begin{array}{cc} A, &  C\\
\phantom{A} & D \end{array}
\Big| \; 1 \right)}_{q}
-\frac{BD(-1)}{q} \binom{A\bar{B}}{\bar{B}}; \, \textup{and}
\\[6pt] \label{for_3F2Red_6}
{_{3}F_2} {\left( \begin{array}{ccc} A, & B, &  C \\
\phantom{A} & D, & ABC\bar{D} \end{array}
\Big| \; 1 \right)}_{q}
&=
BC(-1)\binom{C}{D\bar{A}}\binom{B}{D\bar{C}}
-\frac{BD(-1)}{q} \binom{D\bar{B}}{A}.
\end{align}
We can further reduce the ${_{2}F_1}( \cdot |1)$'s that appear in (\ref{for_3F2Red_1})-(\ref{for_3F2Red_5}) via \cite[Theorem 4.9]{G2},
\begin{equation}\label{for_2F1Red}
{_{2}F_1} {\left( \begin{array}{cc} A, &  B \\
\phantom{A} & C \end{array}
\Big| \; 1 \right)}_{q}
= A(-1) \binom{B}{\bar{A}C}.
\end{equation}
And, we have the following relation which easily follows from their definition,
\begin{equation}\label{for_3F2Per}
{_{3}F_2} {\left( \begin{array}{ccc} A, & B, &  C \\
\phantom{A} & D, & E \end{array}
\Big| \; 1 \right)}_{q}
=
{_{3}F_2} {\left( \begin{array}{ccc} A, & C, &  B \\
\phantom{A} & E, & D \end{array}
\Big| \; 1 \right)}_{q}.
\end{equation}

We also have transformation formulae for finite field hypergeometric functions that relate ${_{3}F_2}( \cdot |1)$ functions with different parameters \cite[Thms 5.14, 5.18 \& 5.20]{G}.  
\begin{align}
\label{for_3F2_T1}
{_{3}F_2} {\left( \begin{array}{ccc} A, & B, &  C \\
\phantom{A} & D, & E \end{array}
\Big| \; 1 \right)}_{q}
&=
{_{3}F_2} {\left( \begin{array}{ccc} B\bar{D}, & A\bar{D}, &  C\bar{D} \\
\phantom{A} & \bar{D}, & E\bar{D} \end{array}
\Big| \; 1 \right)}_{q};
\\[6pt] \label{for_3F2_T2}
{_{3}F_2} {\left( \begin{array}{ccc} A, & B, &  C \\
\phantom{A} & D, & E \end{array}
\Big| \; 1 \right)}_{q}
&=
ABCDE(-1) \,
{_{3}F_2} {\left( \begin{array}{ccc} A, & A\bar{D}, &  A\bar{E} \\
\phantom{A} & A\bar{B}, & A\bar{C} \end{array}
\Big| \; 1 \right)}_{q};
\\[6pt] \label{for_3F2_T3}
{_{3}F_2} {\left( \begin{array}{ccc} A, & B, &  C \\
\phantom{A} & D, & E \end{array}
\Big| \; 1 \right)}_{q}
&=
ABCDE(-1) \,
{_{3}F_2} {\left( \begin{array}{ccc} B\bar{D}, & B, &  B\bar{E} \\
\phantom{A} & B\bar{A}, & B\bar{C} \end{array}
\Big| \; 1 \right)}_{q};
\\[6pt] \label{for_3F2_T4}
{_{3}F_2} {\left( \begin{array}{ccc} A, & B, &  C \\
\phantom{A} & D, & E \end{array}
\Big| \; 1 \right)}_{q}
&=
AE(-1) \,
{_{3}F_2} {\left( \begin{array}{ccc} A, & B, &  \bar{C}E \\
\phantom{A} & AB\bar{D}, & E \end{array}
\Big| \; 1 \right)}_{q};
\\[6pt] \label{for_3F2_T5}
{_{3}F_2} {\left( \begin{array}{ccc} A, & B, &  C \\
\phantom{A} & D, & E \end{array}
\Big| \; 1 \right)}_{q}
&=
AD(-1) \,
{_{3}F_2} {\left( \begin{array}{ccc} A, & D\bar{B}, &  C \\
\phantom{A} & D, & AC\bar{E} \end{array}
\Big| \; 1 \right)}_{q};
\\[6pt] \label{for_3F2_T6}
{_{3}F_2} {\left( \begin{array}{ccc} A, & B, &  C \\
\phantom{A} & D, & E \end{array}
\Big| \; 1 \right)}_{q}
&=
B(-1) \,
{_{3}F_2} {\left( \begin{array}{ccc} \bar{A}D, & B, &  C \\
\phantom{A} & D, & BC\bar{E} \end{array}
\Big| \; 1 \right)}_{q}; \, \textup{and}
\\[6pt] \label{for_3F2_T7}
{_{3}F_2} {\left( \begin{array}{ccc} A, & B, &  C \\
\phantom{A} & D, & E \end{array}
\Big| \; 1 \right)}_{q}
&=
AB(-1) \,
{_{3}F_2} {\left( \begin{array}{ccc} \bar{A}D, & \bar{B}D, &  C \\
\phantom{A} & D, & \bar{AB}DE \end{array}
\Big| \; 1 \right)}_{q}.
\end{align}


\section{Induced subgraphs of $G_k(q)$ and Proofs of Theorems \ref{thm_Main1} and \ref{thm_Main3}}\label{sec_Subgraphs}
In this section, we induce two subgraphs of the $k$-th power Paley digraph $G_k(q)$ and relate the number of transitive subtournaments of a given order for each graph, which we use to prove Theorems \ref{thm_Main1} and \ref{thm_Main3}. 
For a graph $G$, we denote its vertex set by $V(G)$ and its edge set by $E(G)$, so the order of $G$ is $\#V(G)$ and the size of $G$ is $\#E(G)$. For a given vertex $a$ of $G$ we denote the in-degree and out-degree of $a$ in $G$ by $\indeg_G(a)$ and $\outdeg_G(a)$ respectively. 

It is easy to see from its definition that $\#V(G_k(q)) =q$, $\indeg_{G_k(q)}(a) = \outdeg_{G_k(q)}(a) = \frac{q-1}{k}$ for all vertices $a$, and, consequently, $\#E(G_k(q)) = \frac{q(q-1)}{k}$.

Let $H_k(q)$ be the induced subgraph of $G_k(q)$ whose vertex set is $S_k$, the set of $k$-th power residues of $\mathbb{F}_q$, which are the out-neighbors of zero in $G_k(q)$. Therefore, $\#V(H_k(q))=|S_k| = \frac{q-1}{k}$. 
Now 
$$a \to b \in E(H_k(q)) \Longleftrightarrow \chi_k(a)=\chi_k(b)=\chi_k(b-a)=1.$$
So, for $a \in V(H_k(q))$, using (\ref{for_OrthRel}), we get that
\begin{align*}
\outdeg_{H_k(q)}(a) 
&= \frac{1}{k^2} \sum_{b \in \mathbb{F}_q^* \setminus \{a\}} \sum_{s=0}^{k-1} \chi_k^s(b) \sum_{t=0}^{k-1} \chi_k^t(b-a)\\
&= \frac{1}{k^2} \sum_{s,t=0}^{k-1} \chi_k^{t}(-1) \, \chi_k^{s+t}(a) J(\chi_k^s, \chi_k^t)\\
&= \frac{1}{k^2} \, \mathbb{J}_0(q, k)\\
&= \frac{1}{k^2} ( \mathbb{R}_k(q) + q -2k+1) \qquad \textup{(by Prop. \ref{prop_JtoRS})}.
\end{align*}
Similarly, $\indeg_{H_k(q)}(a) = \outdeg_{H_k(q)}(a).$
These degrees are independent of $a$ so 
\begin{align*}
\#E(H_k(q))
&= \#V(H_k(q)) \cdot \outdeg_{H_k(q)}(a)\\
&= \frac{q-1}{k^3} \, \mathbb{J}_0(q, k)\\
&= \frac{q-1}{k^3} \, ( \mathbb{R}_k(q) + q -2k+1) .
\end{align*}

Let $H^{1}_k(q)$ be the induced subgraph of $H_k(q)$ whose vertex set is the set of out-neighbors of 1 in $H_k(q)$. Therefore 
$$\#V(H^{1}_k(q)) = \outdeg_{H_k(q)}(1) =  \frac{1}{k^2} \, \mathbb{J}_0(q, k) = \frac{1}{k^2} ( \mathbb{R}_k(q) + q - 2k+1),$$
and 
$$a \to b \in E(H^{1}_k(q)) \Longleftrightarrow \chi_k(a)=\chi_k(b)=\chi_k(a-1)=\chi_k(b-1)=\chi_k(b-a)=1.$$
Again using (\ref{for_OrthRel}), and noting that $\chi_k(-1)=-1$, we get that for $a \in V(H^{1}_k(q))$,
\begin{align*}
\outdeg_{H^{1}_k(q)}(a) 
&= \frac{1}{k^3} \sum_{b \in \mathbb{F}_q^* \setminus \{1,a\}} \sum_{t_1=0}^{k-1} \chi_k^{t_1}(b) \sum_{t_2=0}^{k-1} \chi_k^{t_2}(b-1) \sum_{t_3=0}^{k-1} \chi_k^{t_3}(b-a)\\
&= \frac{1}{k^3} \sum_{t_1,t_2,t_3=0}^{k-1}   \sum_{b \in \mathbb{F}_q^* \setminus \{1,a\}}  \chi_k^{t_1-t_3}(b) \,  \chi_k^{t_3-t_2}(b-1) \, \chi_k^{-t_1}(b-a)\\
&= \frac{1}{k^3} \sum_{t_1,t_2,t_3=0}^{k-1} 
(-1)^{t_2+t_3} \,
q \, {_{2}F_1} {\left( \begin{array}{ccc} \chi_k^{t_1}, & \chi_k^{t_2} \\[0.05in]
\phantom{\chi_k^{t_1}} & \chi_k^{t_3}  \end{array}
\Big| \; a \right)}_{q}, \qquad \textup{(using (\ref{for_CharSum2F1})),}
\end{align*}
where we have used a change of variables to get the second line. Finally, we get that
\begin{align*}
\#E(H^{1}_k(q))
&= \sum_{a \in V(H^{1}_k(q))} \outdeg_{H^{1}_k(q)}(a) \\
&=  \frac{1}{k^3} \sum_{\substack{a \in \mathbb{F}_q \\ \chi_k(a)=\chi_k(a-1)=1}}  \sum_{t_1,t_2,t_5=0}^{k-1}   \sum_{b \in \mathbb{F}_q^* \setminus \{1,a\}}  \chi_k^{t_1}(b) \,  \chi_k^{t_2}(b-1) \, \chi_k^{t_5}(b-a)\\
&=  \frac{1}{k^5}  \sum_{t_1,t_2,t_3,t_4,t_5=0}^{k-1}  \sum_{\substack{a,b \in \mathbb{F}_q^* \setminus \{1\} \\ a\neq b}}
\chi_k^{t_1}(b) \,  \chi_k^{t_2}(b-1) \, \chi_k^{t_3}(a) \,  \chi_k^{t_4}(a-1) \, \chi_k^{t_5}(b-a) \\
&=  \frac{1}{k^5}  \sum_{t_1,t_2,t_3,t_4,t_5=0}^{k-1}  
\sum_{a,b \in \mathbb{F}_q}
\chi_k^{t_1-t_5}(b) \,  \chi_k^{t_5-t_3}(b-1) \, \chi_k^{t_2}(a) \,  \chi_k^{t_4-t_2}(a-1) \,  \chi_k^{-t_1}(b-a)\\
&=  \frac{1}{k^5}  \sum_{t_1,t_2,t_3,t_4,t_5=0}^{k-1} 
(-1)^{t_3+t_5} \, q^2 \,
{_{3}F_2}\biggl( \begin{array}{ccc} \chi_k^{t_1}, & \chi_k^{t_2}, & \chi_k^{t_3} \vspace{.05in}\\
\phantom{\chi_k^{t_1}} & \chi_k^{t_4}, & \chi_k^{t_5} \end{array}
\Big| \; 1 \biggr)_{q}, \qquad \textup{(using (\ref{for_CharSum3F2})).}
\end{align*}
So we have proved the following proposition.
\begin{prop}\label{prop_Subgraph}
Let $k \geq 2$ be an even integer. Let $q$ be a prime power such that $q \equiv k+1 \imod {2k}$. Let $H_k(q)$ be the induced subgraph of the $k$-th power Paley digraph $G_k(q)$ whose vertex set is the set of $k$-th power residues of $\mathbb{F}_q$. Let $H^{1}_k(q)$ be the induced subgraph of $H_k(q)$ whose vertex set is the set of out-neighbors of 1 in $H_k(q)$. Then
\begin{enumerate}[\bf (a) ]
\item $\#V(H_k(q)) = \frac{q-1}{k};$\\[-6pt]
\item For $a \in V(H_k(q))$, $\indeg_{H_k(q)}(a)  = \outdeg_{H_k(q)}(a) = \frac{1}{k^2} \, \mathbb{J}_0(q, k) = \frac{1}{k^2} ( \mathbb{R}_k(q) + q -2k+1)$;\\
\item $\#E(H_k(q)) = \frac{q-1}{k^3} \, \mathbb{J}_0(q, k) = \frac{q-1}{k^3} \, ( \mathbb{R}_k(q) + q -2k+1)$;\\
\item $\#V(H^{1}_k(q)) = \frac{1}{k^2} \, \mathbb{J}_0(q, k) = \frac{1}{k^2} ( \mathbb{R}_k(q) + q -2k+1)$;\\
\item For $a \in V(H^{1}_k(q))$, $\deg_{H^{1}_k(q)}(a) = \displaystyle \frac{1}{k^3} \sum_{t_1,t_2,t_3=0}^{k-1} 
 (-1)^{t_2+t_3} \, q \, {_{2}F_1} {\left( \begin{array}{ccc} \chi_k^{t_1}, & \chi_k^{t_2} \\[0.05in]
\phantom{\chi_k^{t_1}} & \chi_k^{t_3}  \end{array}
\Big| \; a \right)}_{q}$; and\\
\item $\#E(H^{1}_k(q)) = 
\displaystyle  \frac{1}{k^5} \sum_{t_1,t_2,t_3,t_4,t_5=0}^{k-1}
(-1)^{t_3+t_5} \, q^2
{_{3}F_2}\biggl( \begin{array}{ccc} \chi_k^{t_1}, & \chi_k^{t_2}, & \chi_k^{t_3} \vspace{.05in}\\
\phantom{\chi_k^{t_1}} & \chi_k^{t_4}, & \chi_k^{t_5} \end{array}
\Big| \; 1 \biggr)_{q}$.
\end{enumerate}
\end{prop}

Next we relate the number of transitive subtournaments of a certain order of $G_k(q)$ to those of $H_k(q)$ and $H_k^{1}(q)$.
\begin{lemma}\label{lem_Subgraph}
Let $k, q, G_k(q), H_k(q)$ and $H^{1}_k(q)$ be defined as in Proposition \ref{prop_Subgraph}.
Then, for $m$ a positive integer,
\begin{enumerate}[\bf (a) ]
\item $\mathcal{K}_{m+1}(G_k(q)) = q \, \mathcal{K}_{m}(H_k(q))$; and
\item $\mathcal{K}_{m+1}(H_k(q)) = \frac{q-1}{k} \, \mathcal{K}_{m}(H^{1}_k(q))$.
\end{enumerate}
So for $m \geq 2$
\begin{enumerate}[\bf (c) ]
\item $\mathcal{K}_{m+1}(G_k(q)) = \frac{q(q-1)}{k} \, \mathcal{K}_{m-1}(H^{1}_k(q)).$
\end{enumerate}
\end{lemma}

\begin{proof} A tournament of order $m$ is transitive if and only if the set of out-degrees of its vertices is $\{0,1, \dots ,m-1\}$. We represent a transitive subtournament of order $m$ by the $m$-tuple of its vertices $(a_1, a_2, \cdots, a_m)$, listed in order such that the out-degree of vertex $a_i$ is $m-i$, i.e. the corresponding $m$-tuple of out-degrees is $(m-1, m-2, \cdots, 1, 0)$. And we will call this a transitive subtournament originating from $a_1$.
For a graph $G$, we let $\mathcal{S}_{G, m}$ denote the set of transitive subtournaments of $G$ of order $m$ and we let $\mathcal{S}_{G, m, a}$ denote the set of transitive subtournaments of $G$ of order $m$ originating from $a$.\\
\indent (a) For $a \in V(G_k(q))$, the map $f_a(\lambda)=\lambda+a$ is an automorphism of $G_k(q)$. Thus, $|\mathcal{S}_{G,m+1,a}| = |\mathcal{S}_{G,m+1,0}|$ for all $a \in V(G_k(q))$. Also,
$$(0, a_1,a_2, \cdots,a_m) \in \mathcal{S}_{G, m+1, 0} \Longleftrightarrow  (a_1,a_2, \cdots,a_m)\in \mathcal{S}_{H, m}$$
so $|\mathcal{S}_{G,m+1,0}| = |\mathcal{S}_{H, m}|$. 
Therefore,
$$\mathcal{K}_{m+1}(G_k(q)) = \sum_{a \in V(G_k(q))} |\mathcal{S}_{G,m+1,a}| = q \, |\mathcal{S}_{G,m+1,0}| = q \, |\mathcal{S}_{H}| = q \, \mathcal{K}_{m}(H_k(q)).$$
\indent (b) For $a \in V(H_k(q))$, the map $f_a(\lambda)=a \lambda$ is an automorphism of $H_k(q)$, so $|\mathcal{S}_{H,m+1,a}| = |\mathcal{S}_{H,m+1,1}|$ for all $a \in V(H_k(q))$. Also
$$(1,a_1,a_2, \cdots,a_{m}) \in \mathcal{S}_{H,m+1,1} \Longleftrightarrow (a_1,a_2, \cdots,a_{m}) \in \mathcal{S}_{H^{1}, m}.$$ 
So $|\mathcal{S}_{H,m+1,1}| = |\mathcal{S}_{H^{1},m}|$.  Therefore,
$$\mathcal{K}_{m+1}(H_k(q)) = \sum_{a \in V(H_k(q))} |\mathcal{S}_{H,m+1,a}| = \#V(H_k(q)) |\mathcal{S}_{H^{1}, m}| = \tfrac{q-1}{k} \mathcal{K}_{m}(H^{1}_k(q)).$$
\indent (c) Follows immediately from combining (a) and (b).
\end{proof}
Taking $m=3$ in Lemma \ref{lem_Subgraph} (c) yields Corollary \ref{cor_Subgraph}.
\begin{cor}\label{cor_Subgraph}
Let $k \geq 2$ be an even integer. Let $q$ be a prime power such that $q \equiv k+1 \imod {2k}$.  Then
$$\mathcal{K}_{4}(G_k(q)) = \frac{q(q-1)}{k} \#E(H^{1}_k(q)).$$
\end{cor}
\noindent
Combining Corollary \ref{cor_Subgraph} and Proposition \ref{prop_Subgraph} (f) proves Theorem \ref{thm_Main1}.

Taking $m=2$ in Lemma \ref{lem_Subgraph} (a) yields Corollary \ref{cor_Subgraph2}.
\begin{cor}\label{cor_Subgraph2}
Let $k \geq 2$ be an even integer. Let $q$ be a prime power such that $q \equiv k+1 \imod {2k}$.  Then
$$\mathcal{K}_{3}(G_k(q)) = q \, \#E(H_k(q)).$$
\end{cor}
\noindent
Combining Corollary \ref{cor_Subgraph2} and Proposition \ref{prop_Subgraph} (c) proves Theorem \ref{thm_Main3}.


\section{Proof of Theorem \ref{thm_Main2}}\label{sec_ProofThm2}
We prove Theorem \ref{thm_Main2} by using the reduction formulae (\ref{for_3F2Red_1})-(\ref{for_3F2Red_6}) on the relevant summands in 
\begin{equation}\label{for_3F2_Total}
q^2 \sum_{\vec{t} \in \left({\mathbb{Z}_{k}}\right)^{5}} {_{3}F_2} \left( \vec{t} \; \big| \;  1 \right)_{q,k}
\end{equation}
from Theorem \ref{thm_Main1}.
Combining each of (\ref{for_3F2Red_1})-(\ref{for_3F2Red_6}) with (\ref{for_3F2Per}) yields ten distinct cases.\\[6pt] 
\underline{Case 1 ($t_1=0$):} 
Using (\ref{for_3F2Red_1}), we reduce the summands in (\ref{for_3F2_Total}) which have $t_1 = 0$.
\begin{align*}
q^2 & \sum_{t_2,t_3,t_4,t_5=0}^{k-1} 
(-1)^{t_3+t_5} {_{3}F_2}\biggl( \begin{array}{ccc} \varepsilon, & \chi_k^{t_2}, & \chi_k^{t_3} \vspace{.05in}\\
\phantom{\chi_k^{t_1}} & \chi_k^{t_4}, & \chi_k^{t_5} \end{array}
\Big| \; 1 \biggr)_{q}
\\[6pt]
&=
q^2 \sum_{t_2,t_3,t_4,t_5=0}^{k-1}
(-1)^{t_3+t_5}
\left[
-\frac{1}{q} \,
{_{2}F_1} {\left( \begin{array}{cc} \chi_k^{t_2-t_4}, &  \chi_k^{t_3-t_4} \vspace{.05in}\\
\phantom{\chi_k^{t_2-t_4}} &\chi_k^{t_5-t_4}\end{array}
\Big| \; 1 \right)}_{q}
+ \binom{\chi_k^{t_2}}{\chi_k^{t_4}} \binom{\chi_k^{t_3}}{\chi_k^{t_5}} \right]
\\[6pt]
&=
-q \sum_{t_2,t_3,t_4,t_5=0}^{k-1}
(-1)^{t_3+t_5} 
{_{2}F_1} {\left( \begin{array}{cc} \chi_k^{t_2-t_4}, &  \chi_k^{t_3-t_4} \vspace{.05in}\\
\phantom{\chi_k^{t_2-t_4}} &\chi_k^{t_5-t_4}\end{array}
\Big| \; 1 \right)}_{q} \\
& \qquad \qquad \qquad \qquad \qquad \qquad 
+
\sum_{t_2,t_3,t_4,t_5=0}^{k-1}
(-1)^{t_3+t_4} 
J(\chi_k^{t_2}, \bar{\chi_k}^{t_4}) \, J(\chi_k^{t_3}, \bar{\chi_k}^{t_5})
\\[6pt]
&=
-q \sum_{t_2,t_3,t_4,t_5=0}^{k-1}
(-1)^{t_3+t_5} 
{_{2}F_1} {\left( \begin{array}{cc} \chi_k^{t_2}, &  \chi_k^{t_3} \vspace{.05in}\\
\phantom{\chi_k^{t_2}} &\chi_k^{t_5}\end{array}
\Big| \; 1 \right)}_{q}
+ \mathbb{J}_0(q, k)^2
\\[6pt]
&=
-qk \sum_{t_2,t_3,t_5=0}^{k-1}
(-1)^{t_3+t_2+t_5} 
\binom{\chi_k^{t_3}}{\chi_k^{t_5-t_2}} + \mathbb{J}_0(q, k)^2 \qquad \textup{(using (\ref{for_2F1Red}))}
\\[6pt]
&=
-k \sum_{t_2,t_3,t_5=0}^{k-1}
(-1)^{t_3} 
J(\chi_k^{t_3},\chi_k^{t_2-t_5}) + \mathbb{J}_0(q, k)^2 
\\[6pt]
&=
-k \sum_{t_2,t_3,t_5=0}^{k-1}
(-1)^{t_3} 
J(\chi_k^{t_3},\chi_k^{t_2}) + \mathbb{J}_0(q, k)^2 
\\[6pt]
&=
\mathbb{J}_0(q, k)^2-k^2 \, \mathbb{J}_0(q, k),
\end{align*}
where we have used the fact that $\chi_k(-1)=-1$ in many of the steps.\\[6pt] 
\underline{Case 2 ($t_2=0$):} 
Next, using (\ref{for_3F2Red_2}), we reduce the summands which have $t_2 = 0$, excluding those with $t_1=0$, as they have already been accounted for in Case 1.\\
\begin{align*}
q^2  & \sum_{\substack{t_1,t_3,t_4,t_5=0 \\ t_1 \neq 0}}^{k-1} 
(-1)^{t_3+t_5} 
{_{3}F_2}\biggl( \begin{array}{ccc} \chi_k^{t_1}, & \varepsilon, & \chi_k^{t_3} \vspace{.05in}\\
\phantom{\chi_k^{t_1}} & \chi_k^{t_4}, & \chi_k^{t_5} \end{array}
\Big| \; 1 \biggr)_{q}
\\[6pt]
&=
q^2 \sum_{\substack{t_1,t_3,t_4,t_5=0 \\ t_1 \neq 0}}^{k-1} 
(-1)^{t_3+t_5} 
\left[
(-1)^{t_1} \binom{\chi_k^{t_4}}{\chi_k^{t_1}}
{_{2}F_1} {\left( \begin{array}{cc} \chi_k^{t_1-t_4}, &  \chi_k^{t_3-t_4} \vspace{.05in}\\
\phantom{\chi_k^{t_1-t_4}} & \chi_k^{t_5-t_4} \end{array}
\Big| \; 1 \right)}_{q}
-\frac{(-1)^{t_4}}{q} \binom{\chi_k^{t_3}}{\chi_k^{t_5}}
\right]
\\[6pt] 
&=
\sum_{\substack{t_1,t_3,t_4,t_5=0 \\ t_1 \neq 0}}^{k-1} 
(-1)^{t_3+t_4} 
\left[
J(\chi_k^{t_4},\bar{\chi_k}^{t_1}) \, J(\chi_k^{t_3-t_4},\chi_k^{t_1-t_5})
- J(\chi_k^{t_3},\bar{\chi_k}^{t_5})
\right]  \quad \textup{(using (\ref{for_2F1Red}))}
\\[6pt] 
&=
\sum_{\substack{t_1,t_4=0 \\ t_1 \neq 0}}^{k-1} 
(-1)^{t_4} J(\chi_k^{t_4},\bar{\chi_k}^{t_1}) \, 
\sum_{t_3,t_5=0}^{k-1} 
(-1)^{t_3} J(\chi_k^{t_3-t_4},\chi_k^{t_1-t_5})
- 0
\\[6pt]
&=
\sum_{\substack{t_1,t_4=0 \\ t_1 \neq 0}}^{k-1} 
J(\chi_k^{t_4},\bar{\chi_k}^{t_1}) \, 
\sum_{t_3,t_5=0}^{k-1} 
(-1)^{t_3} J(\chi_k^{t_3},\chi_k^{t_5})
\\[6pt] 
&=
\left[\mathbb{J}_0(q, k)
- \sum_{t=0}^{k-1} J(\chi_k^{t},\varepsilon) \right] 
 \mathbb{J}_0(q, k) 
\\[6pt]
&=
\left[\mathbb{J}_0(q, k)
- (q-2) +(k-1) \right] 
\mathbb{J}_0(q, k) 
\qquad \textup{(using Prop \ref{prop_JacBasic})} 
\\[6pt]
&=
\mathbb{J}_0(q, k)^2
- (q-k-1) \, \mathbb{J}_0(q, k) .
\end{align*}
We evaluate the remaining cases in a similar manner, using only basic properties of Jacobi sums and hypergeometric functions from Sections \ref{sec_Prelim_GJSums} and \ref{sec_Prelim_HypFns}. 
However, these evaluations do become more tedious, as we have to exclude successively more previous cases. We omit the details, for brevity. 
These remaining cases summarize as follows.\\[6pt] 
\underline{Case 3 ($t_3=0$):}
\begin{multline*}
q^2  \sum_{\substack{t_1,t_2,t_4,t_5=0 \\ t_1, t_2 \neq 0}}^{k-1} 
(-1)^{t_5}
{_{3}F_2}\biggl( \begin{array}{ccc} \chi_k^{t_1}, & \chi_k^{t_2}, & \varepsilon \vspace{.05in}\\
\phantom{\chi_k^{t_1}} & \chi_k^{t_4}, & \chi_k^{t_5} \end{array}
\Big| \; 1 \biggr)_{q}\\
=
\mathbb{J}_0(q, k)^2
- (q+k(k-2)) \, \mathbb{J}_0(q, k)
- \mathbb{JJ}_0(q,k) 
+(q-1)(q-k-1).
\end{multline*}
\newpage
\noindent
\underline{Case 4 ($t_4=t_1$):}
\begin{multline*}
q^2  \sum_{\substack{t_1,t_2,t_3,t_5=0 \\ t_1, t_2,t_3 \neq 0}}^{k-1} 
(-1)^{t_3+t_5} 
{_{3}F_2}\biggl( \begin{array}{ccc} \chi_k^{t_1}, & \chi_k^{t_2}, & \chi_k^{t_3} \vspace{.05in}\\
\phantom{\chi_k^{t_1}} & \chi_k^{t_1}, & \chi_k^{t_5} \end{array}
\Big| \; 1 \biggr)_{q}\\
=
\mathbb{J}_0(q, k)^2
- (2q-4k+1) \, \mathbb{J}_0(q, k) 
+(q-k-1)(q-2k+1).
\end{multline*}
\\ \underline{Case 5 ($t_5=t_1$):}
\begin{multline*}
q^2  \sum_{\substack{t_1,t_2,t_3,t_4=0 \\ t_1, t_2,t_3 \neq 0 \\ t_1 \neq t_4}}^{k-1} 
(-1)^{t_3+t_1} 
{_{3}F_2}\biggl( \begin{array}{ccc} \chi_k^{t_1}, & \chi_k^{t_2}, & \chi_k^{t_3} \vspace{.05in}\\
\phantom{\chi_k^{t_1}} & \chi_k^{t_4}, & \chi_k^{t_1} \end{array}
\Big| \; 1 \biggr)_{q}\\
=
\mathbb{J}_0(q, k)^2
- (2q-3k+3) \, \mathbb{J}_0(q, k) 
- \mathbb{JJ}^{-}_0(q,k)
- \mathbb{J}^{-}_0(q,k) 
+2q^2-2(2k-1)q +(k-1)(k+2).
\end{multline*}
\\ \underline{Case 6 ($t_4=t_2$):}
\begin{multline*}
q^2  \sum_{\substack{t_1,t_2,t_3,t_5=0 \\ t_1, t_2,t_3 \neq 0 \\ t_1 \neq t_2, t_5}}^{k-1} 
(-1)^{t_3+t_5}
{_{3}F_2}\biggl( \begin{array}{ccc} \chi_k^{t_1}, & \chi_k^{t_2}, & \chi_k^{t_3} \vspace{.05in}\\
\phantom{\chi_k^{t_1}} & \chi_k^{t_2}, & \chi_k^{t_5} \end{array}
\Big| \; 1 \biggr)_{q}\\
=
\mathbb{J}_0(q, k)^2
- (q+k^2-5k+7) \, \mathbb{J}_0(q, k) 
- 2\, \mathbb{JJ}_0(q,k) 
+2q^2+(k^2-8k+4)q-(k-1)(k^2-4k-2).
\end{multline*}
\\ \underline{Case 7 ($t_5=t_3$):}
\begin{multline*}
q^2  \sum_{\substack{t_1,t_2,t_3,t_4=0 \\ t_1, t_2,t_3 \neq 0 \\ t_1 \neq t_3, t_4 \\ t_2 \neq t_4}}^{k-1} 
{_{3}F_2}\biggl( \begin{array}{ccc} \chi_k^{t_1}, & \chi_k^{t_2}, & \chi_k^{t_3} \vspace{.05in}\\
\phantom{\chi_k^{t_1}} & \chi_k^{t_4}, & \chi_k^{t_3} \end{array}
\Big| \; 1 \biggr)_{q}\\
=
\mathbb{J}_0(q, k)^2
- (2q-5k+9) \, \mathbb{J}_0(q, k) 
- 2\, \mathbb{JJ}^{-}_0(q,k) 
- \mathbb{J}^{-}_0(q,k) 
+3q^2-8(k-1)q+5k^2-10k+1.
\end{multline*}
\\ \underline{Case 8 ($t_5=t_2$):}
\begin{multline*}
q^2  \sum_{\substack{t_1,t_2,t_3,t_4=0 \\ t_1, t_2,t_3 \neq 0 \\ t_1 \neq t_2, t_4 \\ t_2 \neq t_3, t_4}}^{k-1} 
(-1)^{t_3+t_2}
{_{3}F_2}\biggl( \begin{array}{ccc} \chi_k^{t_1}, & \chi_k^{t_2}, & \chi_k^{t_3} \vspace{.05in}\\
\phantom{\chi_k^{t_1}} & \chi_k^{t_4}, & \chi_k^{t_2} \end{array}
\Big| \; 1 \biggr)_{q}\\
=
\mathbb{J}_0(q, k)^2
- (2q-2k+9) \, \mathbb{J}_0(q, k) 
- \mathbb{JJ}_0(q,k) 
- 2\, \mathbb{JJ}^{-}_0(q,k) 
- 3 \, \mathbb{J}^{-}_0(q,k) 
+4q^2-(8k-10)q+k^2-6k+2.
\end{multline*}
\\ \underline{Case 9 ($t_4=t_3$):}
\begin{multline*}
q^2  \sum_{\substack{t_1,t_2,t_3,t_5=0 \\ t_1, t_2,t_3 \neq 0 \\ t_1, t_2 \neq t_3, t_5 \\ t_3 \neq t_5}}^{k-1} 
(-1)^{t_3+t_5} 
{_{3}F_2}\biggl( \begin{array}{ccc} \chi_k^{t_1}, & \chi_k^{t_2}, & \chi_k^{t_3} \vspace{.05in}\\
\phantom{\chi_k^{t_1}} & \chi_k^{t_3}, & \chi_k^{t_5} \end{array}
\Big| \; 1 \biggr)_{q}\\
=
\mathbb{J}_0(q, k)^2
- (2q+k^2-8k+18) \, \mathbb{J}_0(q, k) 
- 3\, \mathbb{JJ}_0(q,k) 
+4q^2+2(k-1)(k-8)q-2(k-2)(k^2-6k+2).
\end{multline*}
\\ \underline{Case 10 ($t_1+t_2+t_3=t_4+t_5$):}
\begin{multline*}
q^2  \sum_{\substack{t_1,t_2,t_3,t_4=0 \\ t_1, t_2, t_3 \neq 0 \\ t_4 \neq t_1, t_2, t_3 \\ t_4 \neq t_1+t_2, t_1+t_3, t_2+t_3}}^{k-1} 
(-1)^{t_1+t_2+t_4}
{_{3}F_2}\biggl( \begin{array}{ccc} \chi_k^{t_1}, & \chi_k^{t_2}, & \chi_k^{t_3} \vspace{.05in}\\
\phantom{\chi_k^{t_1}} & \chi_k^{t_4}, & \chi_k^{t_1+t_2+t_3-t_4} \end{array}
\Big| \; 1 \biggr)_{q}\\
=
\mathbb{J}_0(q, k)^2
- (2q+k^2-10k+24) \, \mathbb{J}_0(q, k) 
- 3\, \mathbb{JJ}_0(q,k) 
+4q^2+2(k^2-10k+11)q-2(k^3-10k^2+21k-7).
\end{multline*}
The total of these ten reducible cases is
\begin{align*}
q^2 \sum_{\vec{t} \in \left({\mathbb{Z}_{k}}\right)^{5} \setminus X_k} {_{3}F_2} \left( \vec{t} \; \big| \;  1 \right)_{q,k}
&=
10 \, \mathbb{J}_0(q, k)^2
-5(3q+k^2-8k+14) \, \mathbb{J}_0(q,k) 
\\[-6pt] & \qquad
-10 \, \mathbb{JJ}_0(q,k) 
-5 \, \mathbb{J}^{-}_0(q,k) 
-5 \, \mathbb{JJ}^{-}_0(q,k) 
\\[6pt] & \qquad \qquad
+21 q^2
+5(k^2 - 14k+12)q
-5k^3+50k^2-85k+21.
\end{align*}
Applying Proposition \ref{prop_JtoRS} yields
\begin{multline*}
q^2 \sum_{\vec{t} \in \left({\mathbb{Z}_{k}}\right)^{5} \setminus X_k} {_{3}F_2} \left( \vec{t} \; \big| \;  1 \right)_{q,k}
=
10 \, \mathbb{R}_k(q)^2 + 5 \, \mathbb{R}_k(q) \left( q-k^2+1\right) -10 \, \mathbb{S}_k(q) 
\\ 
-5 \, \mathbb{S}^{-}_k(q) 
+q^2
-10 (k-1)^2 q
+ 5k^2(k-1)+1.
\end{multline*}
which completes the proof of Theorem \ref{thm_Main2}.


\section{Orbits of $X_k$ and Proofs of Corollaries \ref{cor_k2}, \ref{cor_k4}, \ref{cor_3k2} \& \ref{cor_3k4}}\label{sec_Orbits}
We now want to use Theorem \ref{thm_Main2} to evaluate $\mathcal{K}_4(G_k(q))$ for specific $k$, which requires evaluating the hypergeometric terms in 
\begin{equation}\label{for_3F2_ND}
q^2 \sum_{\vec{t} \in X_k} {_{3}F_2} \left( \vec{t} \; \big| \;  1 \right)_{q,k},
\end{equation}
where 
$X_k := \{  (t_1,t_2,t_3,t_4,t_5) \in \left({\mathbb{Z}_{k}}\right)^{5} \mid t_1,t_2,t_3 \neq 0, t_4,t_5 \, ; \, t_1+t_2+t_3 \neq t_4+t_5 \}$. 
Many of the hypergeometric function summands in (\ref{for_3F2_ND}) can be related via the transformation formulae (\ref{for_3F2Per}), (\ref{for_3F2_T1})-(\ref{for_3F2_T7}). In fact, any two summands related via one of these transformations will be equal, up to sign.
For example, applying (\ref{for_3F2_T4}) we get that
\begin{equation}\label{for_Teg}
{_{3}F_2} \left( (t_1,t_2,t_3,t_4,t_5) \; \big| \;  1 \right)_{q,k}
=
(-1)^{t_1}
{_{3}F_2} \left( (t_1, t_2, t_5-t_3, t_1+t_2-t_4, t_5) \; \big| \;  1 \right)_{q,k}.
\end{equation}
If we ignore the sign (for now), then, to the change in parameters associated to each transformation (\ref{for_3F2Per}), (\ref{for_3F2_T1})-(\ref{for_3F2_T7}), we can associate a map on $X_k$. Continuing the above example, the relation (\ref{for_Teg}) induces the map
$T_4: X_k \to X_k$ given by 
$$T_4(t_1, t_2, t_3, t_4, t_5) = (t_1, t_2, t_5-t_3, t_1+t_2-t_4, t_5),$$
where the addition in each component takes place in $\mathbb{Z}_k$.
Similarly, to the other transformations in (\ref{for_3F2_T1})-(\ref{for_3F2_T7}) and (\ref{for_3F2Per}), we can associate the maps
\begin{align*}
T_1(t_1, t_2, t_3, t_4, t_5) &= (t_2-t_4, t_1-t_4, t_3-t_4, -t_4, t_5-t_4);\\
T_2(t_1, t_2, t_3, t_4, t_5) &= (t_1, t_1-t_4, t_1-t_5, t_1-t_2, t_1-t_3);\\
T_3(t_1, t_2, t_3, t_4, t_5) &= (t_2-t_4, t_2, t_2-t_5, t_2-t_1, t_2-t_3);\\
T_5(t_1, t_2, t_3, t_4, t_5) &= (t_1, t_4-t_2, t_3, t_4, t_1+t_3-t_5);\\
T_6(t_1, t_2, t_3, t_4, t_5) &= (t_4-t_1, t_2, t_3, t_4, t_2+t_3-t_5);\\ 
T_7(t_1, t_2, t_3, t_4, t_5) &= (t_4-t_1, t_4-t_2, t_3, t_4, t_4+t_5-t_1-t_2); \, \textup{and}\\
T_8(t_1, t_2, t_3, t_4, t_5) &= (t_1, t_3, t_2, t_5, t_4).
\end{align*}
We form the group generated by $T_1, T_2, \cdots, T_8$, with operation composition of functions, and call it $\mathbb{G}_k$.
Then $\mathbb{G}_k$ acts on $X_k$. Furthermore, the values of ${_{3}F_2} \left( \vec{t} \; \big| \;  1 \right)_{q,k}$ within an orbit are equal, up to sign. 
The group $\mathbb{G}_k$ and this action has been fully described in \cite{MS} (extending the work in \cite{DM}) and includes python code to generate all the orbits for a given $k$. In particular, $\mathbb{G}_k$ is a group of order $120$ isomorphic to the permutation group $S_5$ and the number of orbits, $N_{\mathbb{G}_k}$,  is given by
\begin{multline*}\label{for_Nk_New}
N_{\mathbb{G}_k}= \frac{1}{120} \left[(k-1)(k^4-9k^3+61k^2-189k+280)
\right. \\ \left.
+
\begin{cases}
0 & \textup{if } k \equiv 1,5,7,11 \imod{12},\\
40k-200 & \textup{if } k \equiv 3,9 \imod{12},\\
105k-180 & \textup{if } k \equiv 2,10 \imod{12},\\
105k-240 & \textup{if }k \equiv 4,8 \imod{12},\\
145k - 380 & \textup{if }k \equiv 6 \imod{12},\\
145k - 440& \textup{if } k \equiv 0 \imod{12}.
\end{cases}
\right]
\end{multline*} 
We note also that
\begin{equation*}\label{for_Xk}
|X_k| = \sum_{t_1,t_2, t_3=1}^{k-1} 
\sum_{\substack{t_4,t_5=0 \\ t_4,t_5 \neq t_1,t_2,t_3 \\ t_4+t_5 \neq t_1+t_2+t_3}}^{k-1} 1
=(k-1)(k^4-9k^3+36k^2-69k+51).
\end{equation*}
We now incorporate the signs associated to the transformations (\ref{for_3F2Per}), (\ref{for_3F2_T1})-(\ref{for_3F2_T7}).
We illustrate how to do this using $k=4$ as a case study.
When $k=4$, there are $|X_4|=93$ summands in (\ref{for_3F2_ND}) but only $N_{\mathbb{G}_4}=6$ orbits, with representatives
$(1,1,1,0,0)^{10}$,
$(3,3,3,0,0)^{10}$,
$(1,3,2,0,0)^{30}$,
$(1,2,2,0,0)^{30}$,
$(1,1,3,0,0)^{12}$, and
$(2,2,2,0,0)^{1}$,
where the superscript represents the number of elements in the orbit.
Within each orbit, we now consider the sign associated to the transformations linking the elements.
For example, the orbit represented by $(1,2,2,0,0)$ contains the elements $(3,2,2,0,0)$, via $T_6$, and $(2,2,2,1,0)$, via $T_3$.
The corresponding transformation of ${_{3}F_2} \left( \vec{t} \; \big| \;  1 \right)_{q,k}$, via (\ref{for_3F2_T6}) and (\ref{for_3F2_T3}) respectively, yields
$${_{3}F_2} \left( (1,2,2,0,0) \; \big| \;  1 \right)_{q,k} = + \, {_{3}F_2} \left( (3,2,2,0,0) \; \big| \;  1 \right)_{q,k}$$
and
$${_{3}F_2} \left( (1,2,2,0,0) \; \big| \;  1 \right)_{q,k} = - \,{_{3}F_2} \left( (2,2,2,1,0) \; \big| \;  1 \right)_{q,k}.$$
Therefore, the net contribution to the sum in (\ref{for_3F2_ND}) for $\vec{t} = (3,2,2,0,0)$ and $\vec{t} = (2,2,2,1,0)$ is zero.
We can relate the values of ${_{3}F_2} \left( \vec{t} \; \big| \;  1 \right)_{q,k}$ for all $\vec{t}$ in the orbit in a similar way, and we find that the net contribution to the sum in (\ref{for_3F2_ND}) for all $30$ elements in the orbit by $(1,2,2,0,0)$ is $+10$.
Doing this for each orbit we find that the six orbits of $X_4$ are 
$(1,1,1,0,0)^{0}$,
$(3,3,3,0,0)^{0}$,
$(1,3,2,0,0)^{0}$,
$(1,2,2,0,0)^{+10}$,
$(1,1,3,0,0)^{0}$, and
$(2,2,2,0,0)^{+1}$,
where, now, the superscripts represent the net contribution of the orbit to the overall sum after signs have been taken into account. 

In general, we apply this approach for any $k$. In fact, we can automate the process by modifying the code provided in \cite{MS} to track the signs and to output the net contribution associated to each orbit.
This code can be found on the first author's webpage.

We now prove Corollaries \ref{cor_k2} and \ref{cor_k4}.
\begin{proof}[Proof of Corollary \ref{cor_k2}]
We apply Theorem \ref{thm_Main2} with $k=2$. Now $\mathbb{R}_2(q) = \mathbb{S}_2(q) = \mathbb{S}^{-}_2(q) = 0$ as there are no indices that satisfy the conditions of the sum in each case. Also, 
$\vec{t} = (1,1,1,0,0)$ is the only element of $X_2$. Therefore, letting $\varphi \in \widehat{\mathbb{F}^{*}_{q}}$ denote the character of order two, 
\begin{equation*}\label{for_k4G2}
\mathcal{K}_4(G_2(q)) =
\frac{q(q-1)}{2^6} 
\Biggl[ 
 q^2 - 10q + 21
-q^2 
{_{3}F_2} {\left( \begin{array}{ccc} \varphi, & \varphi, & \varphi \\
\phantom{\varphi} & \varepsilon, & \varepsilon \end{array}
\Big| \; 1 \right)}_{q}
\Biggr].
\end{equation*}
By \cite[Thm 4.37]{G2} we find that, when $q\equiv 3 \imod{4}$,
\begin{equation*}
q^2 {_{3}F_2} {\left( \begin{array}{ccc} \varphi, & \varphi, & \varphi \\
\phantom{\varphi} & \varepsilon, & \varepsilon \end{array}
\Big| \; 1 \right)}_{q}
=0,
\end{equation*}
which yields the result.
\end{proof}

\begin{proof}[Proof of Corollary \ref{cor_k4}]
Using Proposition \ref{prop_JacXfer} and Lemma \ref{lem_JacOrder4}(1) we get that 
\begin{equation}\label{for_R4}
\mathbb{R}_4(q) = - J(\chi_4,\chi_4) - J(\bar{\chi_4},\bar{\chi_4}) =2x.
\end{equation}
Using Propositions \ref{prop_JacXfer} \& \ref{prop_JacProdq} and Lemma \ref{lem_JacOrder4}(2) we have
\begin{equation}\label{for_S4}
\mathbb{S}_4(q) = -4q + J(\chi_4,\chi_4)^2 + J(\bar{\chi_4},\bar{\chi_4})^2 = 4x^2-6q.
\end{equation}
and
\begin{equation}\label{for_S4M}
\mathbb{S}^{-}_4(q) = - J(\chi_4,\chi_4)^2 - J(\bar{\chi_4},\bar{\chi_4})^2 = 2q - 4x^2.
\end{equation}
As described above, $X_4$ contains six orbits with representatives 
$(1,1,1,0,0)^{0}$,
$(3,3,3,0,0)^{0}$,
$(1,3,2,0,0)^{0}$,
$(1,2,2,0,0)^{+10}$,
$(1,1,3,0,0)^{0}$, and
$(2,2,2,0,0)^{+1}$,
where the superscripts represent the net contribution of the orbit to the overall sum after signs have been taken into account. 
Therefore, taking $k=4$ in Theorem \ref{thm_Main2} and accounting for (\ref{for_R4})-(\ref{for_S4M}) yields
\begin{multline*}
\mathcal{K}_4(G_4(q)) =
\frac{q(q-1)}{2^{12}}. 
\Biggl[ 
q^2+2q(5x-20)
+20x^2-150x+241\\
+10 \, q^2 
{_{3}F_2} {\left( \begin{array}{cccc}  \chi_4, & \varphi, &  \varphi \\
\phantom{\chi_4} & \varepsilon, &  \varepsilon \end{array}
\Big| \; 1 \right)}_{q}
+
q^2 {_{3}F_2} {\left( \begin{array}{ccc} \varphi, & \varphi, & \varphi \\
\phantom{\varphi} & \varepsilon, & \varepsilon \end{array}
\Big| \; 1 \right)}_{q}
\Biggr].
\end{multline*}
From \cite[(6.4)]{DM} we get that, when $q\equiv 1 \imod{4}$,
\begin{equation*}\label{for_3F2_qOno}
q^2 {_{3}F_2} {\left( \begin{array}{ccc} \varphi, & \varphi, & \varphi \\
\phantom{\varphi} & \varepsilon, & \varepsilon \end{array}
\Big| \; 1 \right)}_{q}
=4x^2-2q.
\end{equation*}
which completes the proof.
\end{proof}

We now also have all the ingredients to prove Corollaries \ref{cor_3k2} and \ref{cor_3k4}.

\begin{proof}[Proof of Corollaries \ref{cor_3k2} and \ref{cor_3k4}]
We've seen above that $\mathbb{R}_2(q) = 0$, $\mathbb{R}_4(q) =2x$. Taking $k=2,4$ in Theorem \ref{thm_Main3} yields the results.
\end{proof}


\section{Lower bounds for multicolor directed Ramsey numbers}\label{sec_Ramsey}
Let $t\geq 2$ and $n_1,n_2, \cdots, n_t$ be positive integers. Let $T_m$ denote a tournament of order $m$. The multicolor directed Ramsey number $R(n_1,n_2, \cdots, n_t)$ is the smallest integer $m$, 
such that any $T_m$, whose edges have been colored with $t$ colors, contains a transitive subtournament $T_{n_i}$ in color $i$, for some $1 \leq i \leq t$.
These numbers exist \cite{MT} and, when $t=1$, we recover the usual Ramsey numbers for tournaments.
If $n_1=n_2=\cdots = n_t$, then we use the abbreviated $R_t(n_1)$ to denote the multicolor directed Ramsey number $R(n_1,n_2, \cdots, n_t)$. 
We note that \cite[Prop. 5]{MT}
\begin{equation}\label{for_LBMult}
(R_{t-1}(m)-1)(R(m)-1) +1 \leq R_t(m)  
\end{equation}
Combining with the facts that $R(3)=4$, $R(3,3)=14$ \cite{BD, MT} and $R(4)=8$ \cite{EM}, we get that, for $t \geq 2$,
\begin{equation}\label{for_LBs}
13 \cdot 3^{t-2} + 1 \leq R_t(3)
\qquad \textup{and} \qquad 
7^t+1 \leq R_t(4)
\end{equation}

Recall, $S_k$ is the subgroup of the multiplicative group $\mathbb{F}_q^{\ast}$ of order $\frac{q-1}{k}$ containing the $k$-th power residues, i.e., if $\omega$ is a primitive element of $\mathbb{F}_q$, then $S_k = \langle \omega^k \rangle$. And, our $k$-th power Paley digraph of order $q$, $G_k(q)$, for $q$ a prime power such that $q \equiv k+1 \imod {2k}$, is the graph with vertex set $\mathbb{F}_q$ where $a \to b$ is an edge if and only if $b-a \in S_k$. Recall also, due to the conditions imposed on $q$, that $-1 \notin S_k$.

We now define subsets of $\mathbb{F}_q^{\ast}$, $S_{k,i} := \omega^{i} S_k$, for $0 \leq i \leq \tfrac{k}{2}-1$, and the related directed graphs $G_{k,i}(q)$ with vertex set $\mathbb{F}_q$ where $a \to b$ is an edge if and only if $b-a \in S_{k,i}$. Each $G_{k,i}(q)$ is isomorphic to $G_{k,0}(q)=G_k(q)$, the $k$-th power Paley digraph, via the map $f:V(G_k(q)) \to V(G_{k,i}(q))$ given by $f(a)=\omega^i a$. Now consider the \emph{multicolor $k$-th power Paley tournament} $P_k(q)$ whose vertex set is taken to be $\mathbb{F}_q$ and whose edges are colored in $\frac{k}{2}$ colors according to $a \to b$ has color $i$ if $b-a \in S_{k,i}$. Note that the induced subgraph of color $i$ of $P_k(q)$ is $G_{k,i}(q)$. Thus, $P_k(q)$ has a transitive subtournament $T_m$ in a single color if and only if the $k$-th power Paley digraph contains a transitive subtournament $T_m$.  Therefore, if $\mathcal{K}_m(G_k(q_1))=0$ for some $q_1$, then $q_1<R_{\frac{k}{2}}(m)$. 

So, for $m=3,4$ and a given $k$, we can use Theorems \ref{thm_Main3} and \ref{thm_Main2}, respectively, to search for the greatest $q$ such that $\mathcal{K}_m(G_k(q))=0$, thus establishing that $q<R_{\frac{k}{2}}(m)$.
The $k=2$ and $k=4$ cases for both $m=3$ and $m=4$ can easily be derived from Corollaries \ref{cor_k2}, \ref{cor_k4}, \ref{cor_3k2} and \ref{cor_3k4}, as discussed in Section \ref{sec_Results}.
For higher $k$, our search, of all $q<10000$, yielded the lower bounds shown in Table \ref{tab_LowerBounds}. 
Improvements on known bounds are marked in bold.

\begin{table}[!htbp]
\centering
\begin{tabular}{| c | l | l |}
\hline
$t=\frac{k}{2}$ & $\leq R_{t}(3)$ & $\leq R_{t}(4)$\\[4pt]
\hline
1 & 4 & 8 \\
2 & 14 & \textbf{126} \\
3 & \textbf{44} & 344 \\
4 & 42 & 954 \\
5 & 72 & 3332 \\
\hline
\end{tabular}
\vspace{6pt}
\caption{Lower Bounds for $R_{\frac{k}{2}}(3)$ and $R_{\frac{k}{2}}(4)$. }
\label{tab_LowerBounds}
\end{table}
It is already known that $R(3)=4$, $R(3,3)=14$ and $R(4)=8$. $44 \leq R_3(3)$ and $126 \leq R_2(4)$ improve on the bounds from (\ref{for_LBs}). The remaining bounds in Table \ref{tab_LowerBounds} either equal or fall short of those implied by (\ref{for_LBs}). However, we can combine the bounds in bold with the multiplicative relation (\ref{for_LBMult}) to improve the bounds in  (\ref{for_LBs}), which proves Corollaries \ref{cor_Ramsey4t} and \ref{cor_Ramsey3t}.


\end{document}